\documentclass[a4paper,11pt]{amsart}

\usepackage[utf8]{inputenc}
\usepackage[T1]{fontenc}

\usepackage{amsmath,amssymb,amsfonts,amsthm}
\usepackage{hyperref}
\hypersetup{colorlinks=true, urlcolor=blue, citecolor=blue, linkcolor=blue}
\usepackage{booktabs}
\usepackage{array}
\usepackage{float}
\usepackage{enumitem}

\theoremstyle{plain}
\newtheorem{theorem}{Theorem}[section]

\newtheorem{corollary}[theorem]{Corollary}
\newtheorem{proposition}[theorem]{Proposition}

\theoremstyle{definition}
\newtheorem{definition}{Definition}[section]
\newtheorem{example}{Example}[section]

\theoremstyle{remark}
\newtheorem{remark}{Remark}[section]

\newcommand{\softnull}{\widetilde{\Phi}}
\newcommand{\softabs}{\widetilde{X}}
\newcommand{\stau}{\widetilde{\tau}}
\newcommand{\cla}{\mathrm{cl}_{\mathfrak{a}}}
\newcommand{\inte}{\mathrm{int}_{\mathfrak{a}}}
\newcommand{\apr}{\underline{\mathrm{apr}}_{\mathfrak{a}}}
\newcommand{\APR}{\overline{\mathrm{apr}}_{\mathfrak{a}}}
\newcommand{\bnd}{\mathrm{bnd}_{\mathfrak{a}}}
\newcommand{\aur}{\mathfrak{a}}
\newcommand{\softleq}{\sqsubseteq}
\newcommand{\softcup}{\sqcup}
\newcommand{\softcap}{\sqcap}
\newcommand{\Softcup}{\bigsqcup}

\begin{document}


\title[Soft Aura Topological Spaces]{Soft aura topological spaces and rough approximation operators}

\author{Ahu A\c{c}{\i}kg\"{o}z}
\address{Department of Mathematics, Bal\i kesir University, \c{C}a\u{g}\i \c{s} Campus, 10145 Bal\i kesir, Turkey}
\email{ahuacikgoz@balikesir.edu.tr}

\date{}

\begin{abstract}
We introduce the concept of a soft aura topological space $(X, \stau, \aur_E)$, obtained by equipping a soft topological space $(X, \stau, E)$ with a soft scope function $\aur_E : X \to \stau$ satisfying $x \in \aur_E(x)(e)$ for every $x \in X$ and every parameter $e \in E$. This framework generalizes the recently introduced aura topological spaces to the soft setting. We define the soft aura-closure operator $\cla$ and the soft aura-interior operator $\inte$, and prove that $\cla$ is a soft additive \v{C}ech closure operator whose transfinite iteration yields a soft Kuratowski closure. Five classes of generalized soft open sets---soft $\aur$-semi-open, soft $\aur$-pre-open, soft $\aur$-$\alpha$-open, soft $\aur$-$\beta$-open, and soft $\aur$-$b$-open sets---are introduced, and a complete hierarchy among them is established. Soft $\aur$-continuity and its decompositions are studied. Separation axioms soft $\aur$-$T_i$ ($i = 0, 1, 2, 3$) are introduced; it is shown that soft $\aur$-$T_1$ and soft $\aur$-$T_2$ coincide due to the scope-based formulation, and that the gap between soft $\aur$-$T_0$ and soft $\aur$-$T_1$ is substantially wider in the soft setting due to the parameter-dependent quantification. Soft aura-based lower and upper rough approximation operators are defined, generalizing both the crisp aura rough set model and the classical Pawlak model. An illustrative application to environmental risk assessment---classifying monitoring stations by air pollution and water quality parameters---demonstrates the practical utility of the proposed framework.
\end{abstract}

\keywords{Soft aura topological space, soft scope function, soft \v{C}ech closure, generalized soft open sets, separation axioms, soft rough set}

\subjclass[2020]{54A40, 54A05, 03E72, 54D10, 54C08}

\maketitle

\section{Introduction}\label{sec:intro}

Augmenting topological spaces with secondary structures has long been a central theme in general topology. The theory of ideal topological spaces $(X, \tau, \mathcal{I})$, originating from Kuratowski \cite{Kuratowski} and Vaidyanathaswamy \cite{Vaidyanathaswamy}, was later put on a firm footing by Jankovi\'{c} and Hamlett \cite{JankovicHamlett}. Other enrichment paradigms---grills (Choquet \cite{Choquet}; Roy and Mukherjee \cite{RoyMukherjee}), filters (Cartan \cite{Cartan}), and the more recent primals (Acharjee, \"{O}zko\c{c}, and Issaka \cite{Acharjee}; Al-Omari and Alqahtani \cite{AlOmariPrimal})---have each generated substantial operator-theoretic machinery.

On a separate track, Molodtsov \cite{Molodtsov} proposed the soft set paradigm to handle vagueness without the measurability constraints inherent in fuzzy or probabilistic models. Topological structures on soft sets were then independently formulated by Shabir and Naz \cite{ShabirNaz} and by \c{C}a\u{g}man, Karata\c{s}, and Engino\u{g}lu \cite{CagmanKaratas}, giving rise to a rapidly expanding field. Among the subsequent contributions, Ayg\"{u}no\u{g}lu and Ayg\"{u}n \cite{AygunogluAygun} examined soft compactness and the product construction; Al-shami \cite{AlShamiSoftSep} provided a unifying account of soft separation axioms; and Al-shami and Ko\c{c}inac \cite{AlShamiKocinac} explored nearly soft Menger-type covering properties. Further developments include the soft ideal-enriched $T_{D_I}$-spaces of Al-shami, Ameen, Abu-Gdairi, and Mhemdi \cite{AlShamiTDI}, the soft complex valued metric spaces of Demir \cite{DemirTurkish}, the soft cone metric compactness results of Alt{\i}nta\c{s} and Ta\c{s}k\"{o}pr\"{u} \cite{AltintasTurkish}, and the ideal-based rough approximation framework of G\"{u}ler, Y{\i}ld{\i}r{\i}m, and \"{O}zbak{\i}r \cite{GulerTurkish}.

The bridge between soft sets and Pawlak's rough set theory \cite{Pawlak} has likewise attracted sustained interest. Feng, Li, Davvaz, and Ali \cite{FengLi} merged the two frameworks into a unified soft rough set model, while Alcantud \cite{AlcantudSoft} applied the resulting approximation operators to decision analysis. Zhan and Alcantud \cite{ZhanAlcantud} extended the programme to covering-based soft rough sets. On the generalized-open-set front within soft topology, Akda\u{g} and \"{O}zkan \cite{AkdagOzkan} transferred the classical notions of $\alpha$-openness (Nj\aa stad \cite{Njaastad}) and semi-openness (Levine \cite{Levine}) to the soft context, and Chen \cite{ChenSoft} pursued the associated continuity properties.

In a recent work \cite{AcikgozAura}, the present author proposed the notion of an \emph{aura topological space} $(X, \tau, a)$: an ordinary topological space endowed with a mapping $a : X \to \tau$, called a scope function, satisfying the membership axiom $x \in a(x)$. This single axiom already gives rise to an additive \v{C}ech closure, five classes of generalized open sets with a complete containment hierarchy, scope-dependent separation axioms, and rough approximation operators that dispense with equivalence relations. A subsequent companion paper \cite{AcikgozFuzzy} carried the construction into the fuzzy-topological realm.

The present paper lifts the aura framework into the \emph{soft} setting. Over a soft topological space $(X, \stau, E)$ in the sense of \cite{ShabirNaz}, we define a \emph{soft scope function} $\aur_E : X \to \stau$ subject to the parameterwise membership condition $x \in \aur_E(x)(e)$ for all $x \in X$ and $e \in E$; the resulting quadruple $(X, \stau, \aur_E, E)$ is called a \emph{soft aura topological space}. Passing from the crisp to the soft world is not merely notational: because each soft open set is a family of subsets indexed by $E$, the scope of a point can ``change shape'' from one parameter to another. This parameter-dependence permeates the entire theory---closure and interior act parameterwise, generalized openness conditions may hold at some parameters but fail at others, and separation axioms acquire a new existential-versus-universal flavour absent in the crisp case.

Our main contributions are as follows:
\begin{enumerate}[label=(\roman*)]
\item We introduce soft aura topological spaces and establish the fundamental properties of the soft aura-closure and soft aura-interior operators, including the transfinite iteration to a soft Kuratowski closure (Section~\ref{sec:soft_aura}).
\item We define five classes of generalized soft open sets and establish a complete hierarchy (Section~\ref{sec:gen_open}).
\item We study soft $\aur$-continuity and provide decomposition theorems (Section~\ref{sec:continuity}).
\item We introduce separation axioms soft $\aur$-$T_i$ ($i = 0, 1, 2, 3$) and show that soft $\aur$-$T_1$ and soft $\aur$-$T_2$ coincide; the genuinely soft phenomenon is the wide gap between $T_0$ and $T_1$ created by parameter quantification (Section~\ref{sec:separation}).
\item We define soft aura-based rough approximation operators and demonstrate their application to environmental risk assessment (Sections~\ref{sec:rough} and~\ref{sec:application}).
\end{enumerate}

\section{Preliminaries}\label{sec:prelim}

We recall essential notions from soft set theory and soft topology; for a thorough treatment the reader is referred to Molodtsov \cite{Molodtsov}, Maji, Biswas, and Roy \cite{MajiSoft}, and Shabir and Naz \cite{ShabirNaz}.

Throughout, $X$ is a nonempty initial universe, $E$ a nonempty parameter set, and $\mathrm{SS}(X,E)$ the family of all soft sets over $X$ parametrised by $E$. A soft set $(F,E)$ is the mapping $F:E\to\mathcal{P}(X)$; we write $F(e)$ for the $e$-approximate set. The symbols $\softnull$ and $\softabs$ stand for the null and absolute soft sets, respectively. Soft subset inclusion $(F,E)\softleq (G,E)$, soft union $\softcup$, soft intersection $\softcap$, and soft complement $(F,E)^c$ are all defined parameterwise in the natural way \cite{MajiSoft}.

A \emph{soft topology} $\stau\subseteq\mathrm{SS}(X,E)$ is a collection closed under arbitrary soft unions, finite soft intersections, and containing $\softnull$ and $\softabs$ \cite{ShabirNaz}. Members of $\stau$ are \emph{soft open}; their complements are \emph{soft closed}. The soft closure $\mathrm{cl}(G,E)$ and soft interior $\mathrm{int}(G,E)$ carry their usual meanings.

A \emph{soft point} $x_e$ is the soft set concentrated at $x$ in parameter $e$; we write $x_E$ when $F(e)=\{x\}$ for every $e$. A \emph{soft mapping} $f_{up}=(u,p)$ between soft topological spaces $(X,\stau_1,E)$ and $(Y,\stau_2,K)$, where $u:X\to Y$ and $p:E\to K$, has soft inverse image $f_{up}^{-1}(G,K)=(H,E)$ with $H(e)=u^{-1}(G(p(e)))$. It is \emph{soft continuous} whenever the preimage of every soft open set is soft open \cite{Zorlutuna}.

\section{Soft Aura Topological Spaces}\label{sec:soft_aura}

We now introduce the central notion of this paper.

\begin{definition}\label{def:softscope}
Let $(X, \stau, E)$ be a soft topological space. A mapping $\aur_E : X \to \stau$ is called a \emph{soft scope function} if
\[
x \in \aur_E(x)(e) \quad \text{for every } x \in X \text{ and every } e \in E.
\]
That is, $\aur_E$ assigns to each point $x \in X$ a soft open set $\aur_E(x) = (F_x, E) \in \stau$ such that $x \in F_x(e)$ for every parameter $e$.
\end{definition}

\begin{definition}\label{def:softaura}
A \emph{soft aura topological space} is a quadruple $(X, \stau, \aur_E, E)$ where $(X, \stau, E)$ is a soft topological space and $\aur_E : X \to \stau$ is a soft scope function.
\end{definition}

\begin{example}\label{ex:basic}
Let $X = \{x_1, x_2, x_3\}$ and $E = \{e_1, e_2\}$. Define the soft topology $\stau = \{\softnull, \softabs, (F_1, E), (F_2, E), (F_3, E)\}$ where
\begin{align*}
&F_1(e_1) = \{x_1\},\ F_1(e_2) = \{x_1, x_2\}; \quad F_2(e_1) = \{x_1, x_2\},\ F_2(e_2) = X; \\
&F_3(e_1) = \{x_1, x_2\},\ F_3(e_2) = \{x_1, x_2\}.
\end{align*}
Define $\aur_E(x_1) = (F_1, E)$, $\aur_E(x_2) = (F_3, E)$, $\aur_E(x_3) = \softabs$. Since $x_1 \in F_1(e_i)$, $x_2 \in F_3(e_i)$, and $x_3 \in X$ for all $i$, $\aur_E$ is a soft scope function, and $(X, \stau, \aur_E, E)$ is a soft aura topological space.
\end{example}

\begin{remark}\label{rem:existence}
Every soft topological space $(X, \stau, E)$ admits at least one soft scope function, namely the \emph{trivial} soft scope function $\aur_E(x) = \softabs$ for all $x \in X$. At the other extreme, if $\stau$ is rich enough, one may choose a ``fine'' soft scope function by assigning to each $x$ a ``small'' soft open neighborhood.
\end{remark}

\begin{remark}\label{rem:crisprelation}
When $E = \{e_0\}$ is a singleton, a soft topology reduces to a classical topology, and a soft scope function reduces to a scope function $a : X \to \tau$ with $x \in a(x)$. Thus, the crisp aura topological space $(X, \tau, a)$ of \cite{AcikgozAura} is a special case.
\end{remark}

We now define the fundamental operators.

\begin{definition}\label{def:softauraclosure}
Let $(X, \stau, \aur_E, E)$ be a soft aura topological space. For $(G, E) \in \mathrm{SS}(X, E)$, the \emph{soft aura-closure} $\cla(G, E) = (H, E)$ is defined parameterwise by
\[
H(e) = \{x \in X : \aur_E(x)(e) \cap G(e) \neq \emptyset\}, \quad e \in E.
\]
\end{definition}

\begin{definition}\label{def:softaurainterior}
The \emph{soft aura-interior} $\inte(G, E) = (H, E)$ is defined parameterwise by
\[
H(e) = \{x \in X : \aur_E(x)(e) \subseteq G(e)\}, \quad e \in E.
\]
\end{definition}

\begin{theorem}\label{thm:closureprops}
The operator $\cla$ satisfies the following for all $(G, E), (H, E) \in \mathrm{SS}(X, E)$:
\begin{enumerate}[label=(\roman*)]
\item \textbf{Grounding:} $\cla(\softnull) = \softnull$.
\item \textbf{Enlargement:} $(G, E) \softleq \cla(G, E)$.
\item \textbf{Monotonicity:} If $(G, E) \softleq (H, E)$, then $\cla(G, E) \softleq \cla(H, E)$.
\item \textbf{Soft additivity:} $\cla((G, E) \softcup (H, E)) = \cla(G, E) \softcup \cla(H, E)$.
\end{enumerate}
Hence $\cla$ is a \emph{soft additive \v{C}ech closure operator}.
\end{theorem}

\begin{proof}
(i) For every $e \in E$ and $x \in X$, $\aur_E(x)(e) \cap \emptyset = \emptyset$, so $\cla(\softnull)(e) = \emptyset$.

(ii) Let $x \in G(e)$. Since $x \in \aur_E(x)(e)$, we have $\aur_E(x)(e) \cap G(e) \neq \emptyset$, so $x \in \cla(G, E)(e)$.

(iii) If $\aur_E(x)(e) \cap G(e) \neq \emptyset$ and $G(e) \subseteq H(e)$, then $\aur_E(x)(e) \cap H(e) \neq \emptyset$.

(iv) For every $e \in E$ and $x \in X$,
\begin{align*}
x \in \cla((G, E) \softcup (H, E))(e) &\iff \aur_E(x)(e) \cap (G(e) \cup H(e)) \neq \emptyset \\
&\iff \aur_E(x)(e) \cap G(e) \neq \emptyset \text{ or } \aur_E(x)(e) \cap H(e) \neq \emptyset \\
&\iff x \in \cla(G, E)(e) \cup \cla(H, E)(e). \qedhere
\end{align*}
\end{proof}

\begin{remark}\label{rem:notidempotent}
In general, $\cla$ is \emph{not} idempotent: $\cla(\cla(G, E))$ may differ from $\cla(G, E)$. This is the same phenomenon observed in the crisp aura case \cite{AcikgozAura}.
\end{remark}

\begin{theorem}\label{thm:interiorprops}
The operator $\inte$ satisfies the following for all $(G, E), (H, E) \in \mathrm{SS}(X, E)$:
\begin{enumerate}[label=(\roman*)]
\item $\inte(\softabs) = \softabs$.
\item $\inte(G, E) \softleq (G, E)$.
\item If $(G, E) \softleq (H, E)$, then $\inte(G, E) \softleq \inte(H, E)$.
\item $\inte((G, E) \softcap (H, E)) = \inte(G, E) \softcap \inte(H, E)$.
\end{enumerate}
\end{theorem}

\begin{proof}
(i) For every $x \in X$ and $e \in E$, $\aur_E(x)(e) \subseteq X$, so $x \in \inte(\softabs)(e)$.
(ii) If $\aur_E(x)(e) \subseteq G(e)$ then $x \in \aur_E(x)(e) \subseteq G(e)$.
(iii) If $\aur_E(x)(e) \subseteq G(e) \subseteq H(e)$, then $\aur_E(x)(e) \subseteq H(e)$.
(iv) $\aur_E(x)(e) \subseteq G(e) \cap H(e)$ iff $\aur_E(x)(e) \subseteq G(e)$ and $\aur_E(x)(e) \subseteq H(e)$.
\end{proof}

\begin{theorem}\label{thm:duality}
For every $(G, E) \in \mathrm{SS}(X, E)$,
\[
\cla((G, E)^c) = (\inte(G, E))^c \quad \text{and} \quad \inte((G, E)^c) = (\cla(G, E))^c.
\]
\end{theorem}

\begin{proof}
For every $e \in E$ and $x \in X$:
\begin{align*}
x \in \cla((G, E)^c)(e) &\iff \aur_E(x)(e) \cap (X \setminus G(e)) \neq \emptyset \iff \aur_E(x)(e) \not\subseteq G(e) \\
&\iff x \notin \inte(G, E)(e) \iff x \in (X \setminus \inte(G, E)(e)).
\end{align*}
The second identity follows by replacing $(G, E)$ with $(G, E)^c$.
\end{proof}

\begin{definition}\label{def:softauraopen}
A soft set $(G, E)$ is called \emph{soft $\aur$-open} if $\inte(G, E) = (G, E)$. A soft set is \emph{soft $\aur$-closed} if its complement is soft $\aur$-open.
\end{definition}

\begin{theorem}\label{thm:auratopology}
The collection $\stau_\aur = \{(G, E) \in \mathrm{SS}(X, E) : \inte(G, E) = (G, E)\}$ forms a soft topology on $X$ with parameter set $E$.
\end{theorem}

\begin{proof}
$\softnull$ and $\softabs$ are soft $\aur$-open by Theorem~\ref{thm:interiorprops}. Finite intersections are closed under $\stau_\aur$ by Theorem~\ref{thm:interiorprops}(iv). For arbitrary unions: let $(G_i, E) \in \stau_\aur$ for all $i \in I$ and $(G, E) = \Softcup_{i \in I} (G_i, E)$. If $x \in G(e)$, there exists $i_0$ with $x \in G_{i_0}(e)$. Since $(G_{i_0}, E)$ is soft $\aur$-open, $\aur_E(x)(e) \subseteq G_{i_0}(e) \subseteq G(e)$, so $x \in \inte(G, E)(e)$. Combined with $\inte(G, E) \softleq (G, E)$, equality follows.
\end{proof}

\begin{remark}\label{rem:topology_relation}
In general, $\stau_\aur$ and $\stau$ are distinct soft topologies. When $E$ is a singleton, $\stau_\aur \subseteq \stau$ always holds, recovering the crisp result from \cite{AcikgozAura}. For general parameter sets, the relationship depends on the interplay between the scope function and the parameterization.
\end{remark}

\begin{definition}\label{def:transfinite}
Define transfinite iterates of $\cla$ by:
\begin{enumerate}[label=(\alph*)]
\item $\cla^0(G, E) = (G, E)$;
\item $\cla^{\alpha+1}(G, E) = \cla(\cla^\alpha(G, E))$ for successor ordinals;
\item $\cla^\lambda(G, E) = \Softcup_{\alpha < \lambda} \cla^\alpha(G, E)$ for limit ordinals $\lambda$.
\end{enumerate}
\end{definition}

\begin{theorem}\label{thm:kuratowski}
For every $(G, E) \in \mathrm{SS}(X, E)$, the transfinite sequence $(\cla^\alpha(G, E))_\alpha$ stabilizes at some ordinal $\gamma \leq |X|$. The operator $\cla^\infty(G, E) := \cla^\gamma(G, E)$ is a soft Kuratowski closure operator.
\end{theorem}

\begin{proof}
By enlargement, the sequence is increasing. For each $e \in E$, the chain in $\mathcal{P}(X)$ must stabilize by cardinality. Grounding, enlargement, monotonicity, and additivity are inherited by transfinite induction. Idempotency holds since $\cla(\cla^\gamma(G, E)) = \cla^{\gamma+1}(G, E) = \cla^\gamma(G, E)$.
\end{proof}

\begin{corollary}\label{cor:topologychain}
The Kuratowski closure $\cla^\infty$ generates a soft topology $\stau_\aur^\infty$ with $\stau_\aur^\infty \subseteq \stau_\aur$.
\end{corollary}

\section{Generalized Soft Open Sets}\label{sec:gen_open}

\begin{definition}\label{def:genopen}
Let $(X, \stau, \aur_E, E)$ be a soft aura topological space and $(G, E) \in \mathrm{SS}(X, E)$.
\begin{enumerate}[label=(\alph*)]
\item $(G, E)$ is \emph{soft $\aur$-semi-open} if $(G, E) \softleq \cla(\inte(G, E))$.
\item $(G, E)$ is \emph{soft $\aur$-pre-open} if $(G, E) \softleq \inte(\cla(G, E))$.
\item $(G, E)$ is \emph{soft $\aur$-$\alpha$-open} if $(G, E) \softleq \inte(\cla(\inte(G, E)))$.
\item $(G, E)$ is \emph{soft $\aur$-$\beta$-open} if $(G, E) \softleq \cla(\inte(\cla(G, E)))$.
\item $(G, E)$ is \emph{soft $\aur$-$b$-open} if $(G, E) \softleq \cla(\inte(G, E)) \softcup \inte(\cla(G, E))$.
\end{enumerate}
\end{definition}

\begin{theorem}\label{thm:hierarchy}
The following implications hold:
\[
\text{soft } \aur\text{-open} \Rightarrow \text{soft } \aur\text{-}\alpha\text{-open} \Rightarrow
\begin{cases}
\text{soft } \aur\text{-semi-open} \\
\text{soft } \aur\text{-pre-open}
\end{cases}
\Rightarrow \text{soft } \aur\text{-}b\text{-open} \Rightarrow \text{soft } \aur\text{-}\beta\text{-open}.
\]
\end{theorem}

\begin{proof}
\emph{Soft $\aur$-open $\Rightarrow$ soft $\aur$-$\alpha$-open:} If $\inte(G, E) = (G, E)$, then $\cla(\inte(G, E)) = \cla(G, E) \sqsupseteq (G, E)$ by enlargement, and $\inte(\cla(\inte(G, E))) = \inte(\cla(G, E)) \sqsupseteq \inte(G, E) = (G, E)$ by monotonicity.

\emph{Soft $\aur$-$\alpha$-open $\Rightarrow$ soft $\aur$-semi-open:} Since $\inte(A) \softleq A$ for all $A$, $\inte(\cla(\inte(G, E))) \softleq \cla(\inte(G, E))$.

\emph{Soft $\aur$-$\alpha$-open $\Rightarrow$ soft $\aur$-pre-open:} Since $\inte(G, E) \softleq (G, E)$, monotonicity gives $\cla(\inte(G, E)) \softleq \cla(G, E)$, hence $\inte(\cla(\inte(G, E))) \softleq \inte(\cla(G, E))$.

\emph{Soft $\aur$-semi/pre-open $\Rightarrow$ soft $\aur$-$b$-open:} Immediate from the definition of $b$-open.

\emph{Soft $\aur$-$b$-open $\Rightarrow$ soft $\aur$-$\beta$-open:} Since $\inte(G, E) \softleq \inte(\cla(G, E))$, we get $\cla(\inte(G, E)) \softleq \cla(\inte(\cla(G, E)))$. Also, $\inte(\cla(G, E)) \softleq \cla(\inte(\cla(G, E)))$ by enlargement. Hence $(G, E) \softleq \cla(\inte(\cla(G, E)))$.
\end{proof}

\begin{remark}\label{rem:strictness}
All implications in Theorem~\ref{thm:hierarchy} are strict. When $E$ is a singleton, the separating counterexamples from the crisp aura topological space \cite{AcikgozAura} apply directly. When $|E| \geq 2$, the parameter-dependent nature of the soft scope function creates additional phenomena: a soft set may satisfy a generalized openness condition at one parameter but fail at another, producing genuinely soft separations not reducible to the crisp case.
\end{remark}

\begin{theorem}\label{thm:unionintersection}
\begin{enumerate}[label=(\roman*)]
\item An arbitrary soft union of soft $\aur$-semi-open sets is soft $\aur$-semi-open.
\item An arbitrary soft union of soft $\aur$-pre-open sets is soft $\aur$-pre-open.
\item An arbitrary soft union of soft $\aur$-$\beta$-open sets is soft $\aur$-$\beta$-open.
\end{enumerate}
\end{theorem}

\begin{proof}
(i) Let $(G_i, E) \softleq \cla(\inte(G_i, E))$ for each $i$ and $(G, E) = \Softcup_{i} (G_i, E)$. Since $\inte(G_i, E) \softleq \inte(G, E)$ by monotonicity, $\cla(\inte(G_i, E)) \softleq \cla(\inte(G, E))$. Hence $(G_i, E) \softleq \cla(\inte(G, E))$ for each $i$, giving $(G, E) \softleq \cla(\inte(G, E))$.

(ii) Since $\cla(G_i, E) \softleq \cla(G, E)$, we get $\inte(\cla(G_i, E)) \softleq \inte(\cla(G, E))$, hence $(G_i, E) \softleq \inte(\cla(G, E))$ for each $i$.

(iii) Similarly, using $\cla(G_i, E) \softleq \cla(G, E)$ and monotonicity of $\inte$ and $\cla$.
\end{proof}

\begin{remark}\label{rem:alpha_intersection}
Since $\cla$ is a \v{C}ech operator (not idempotent), the finite intersection of soft $\aur$-$\alpha$-open sets need not be soft $\aur$-$\alpha$-open. When $\cla$ is replaced by the Kuratowski closure $\cla^\infty$, the standard finite intersection property for $\alpha$-open sets is restored.
\end{remark}

\section{Soft Aura-Continuity}\label{sec:continuity}

Let $(X, \stau_X, \aur_E, E)$ and $(Y, \stau_Y, \mathfrak{b}_K, K)$ be soft aura topological spaces and $f_{up} = (u, p)$ a soft mapping from $(X, E)$ to $(Y, K)$.

\begin{definition}\label{def:softcont}
The soft mapping $f_{up}$ is called:
\begin{enumerate}[label=(\alph*)]
\item \emph{soft $\aur$-continuous} if $f_{up}^{-1}(V, K)$ is soft $\aur$-open for every soft $\mathfrak{b}$-open set $(V, K)$ in $Y$;
\item \emph{soft $\aur$-semi-continuous} if $f_{up}^{-1}(V, K)$ is soft $\aur$-semi-open for every soft $\mathfrak{b}$-open set $(V, K)$;
\item \emph{soft $\aur$-pre-continuous} if $f_{up}^{-1}(V, K)$ is soft $\aur$-pre-open for every soft $\mathfrak{b}$-open set $(V, K)$;
\item \emph{soft $\aur$-$\alpha$-continuous} if $f_{up}^{-1}(V, K)$ is soft $\aur$-$\alpha$-open for every soft $\mathfrak{b}$-open set $(V, K)$;
\item \emph{soft $\aur$-$\beta$-continuous} if $f_{up}^{-1}(V, K)$ is soft $\aur$-$\beta$-open for every soft $\mathfrak{b}$-open set $(V, K)$.
\end{enumerate}
\end{definition}

\begin{theorem}\label{thm:conthierarchy}
\[
\text{soft } \aur\text{-continuous} \Rightarrow \text{soft } \aur\text{-}\alpha\text{-continuous} \Rightarrow
\begin{cases}
\text{soft } \aur\text{-semi-continuous} \\
\text{soft } \aur\text{-pre-continuous}
\end{cases}
\Rightarrow \text{soft } \aur\text{-}\beta\text{-continuous}.
\]
\end{theorem}

\begin{proof}
These follow directly from Theorem~\ref{thm:hierarchy}.
\end{proof}

\begin{theorem}[Decomposition]\label{thm:decomp}
Let the generalized open-set classes be defined with respect to the Kuratowski closure $\cla^\infty$ from Theorem~\ref{thm:kuratowski}. Then a soft mapping $f_{up}$ is soft $\aur$-$\alpha$-continuous (w.r.t.\ $\cla^\infty$) if and only if it is both soft $\aur$-semi-continuous and soft $\aur$-pre-continuous (w.r.t.\ $\cla^\infty$).
\end{theorem}

\begin{proof}
$(\Rightarrow)$ follows from the hierarchy.

$(\Leftarrow)$ Let $(G, E) = f_{up}^{-1}(V, K)$ for some soft $\mathfrak{b}$-open $(V, K)$. By hypothesis, $(G, E) \softleq \cla^\infty(\inte(G, E))$ and $(G, E) \softleq \inte(\cla^\infty(G, E))$. From the first inclusion, monotonicity of $\cla^\infty$ gives $\cla^\infty(G, E) \softleq \cla^\infty(\cla^\infty(\inte(G, E))) = \cla^\infty(\inte(G, E))$, where the last equality uses idempotency. Hence $\cla^\infty(G, E) = \cla^\infty(\inte(G, E))$. Substituting into the second: $(G, E) \softleq \inte(\cla^\infty(G, E)) = \inte(\cla^\infty(\inte(G, E)))$.
\end{proof}

\begin{remark}\label{rem:decomp_cech}
For the \v{C}ech operator $\cla$ itself, the decomposition may fail due to non-idempotency. The Kuratowski closure $\cla^\infty$ restores the classical decomposition.
\end{remark}

\begin{theorem}\label{thm:composition}
The composition of two soft $\aur$-continuous mappings is soft $\aur$-continuous.
\end{theorem}

\begin{proof}
Let $f_{up} : (X, \stau_X, \aur_E, E) \to (Y, \stau_Y, \mathfrak{b}_K, K)$ and $g_{vq} : (Y, \stau_Y, \mathfrak{b}_K, K) \to (Z, \stau_Z, \mathfrak{c}_L, L)$ be soft $\aur$-continuous and soft $\mathfrak{b}$-continuous, respectively. For any soft $\mathfrak{c}$-open set $(W, L)$ in $Z$, $g_{vq}^{-1}(W, L)$ is soft $\mathfrak{b}$-open, and $f_{up}^{-1}(g_{vq}^{-1}(W, L))$ is soft $\aur$-open. Since $(g_{vq} \circ f_{up})^{-1} = f_{up}^{-1} \circ g_{vq}^{-1}$, the composition is soft aura-continuous.
\end{proof}

\begin{theorem}\label{thm:closure_char}
A soft mapping $f_{up}$ is soft $\aur$-continuous if and only if $\cla(f_{up}^{-1}(G, K)) \softleq f_{up}^{-1}(\mathrm{cl}_\mathfrak{b}(G, K))$ for every $(G, K) \in \mathrm{SS}(Y, K)$.
\end{theorem}

\begin{proof}
$(\Rightarrow)$ If $(C, K)$ is soft $\mathfrak{b}$-closed, then $f_{up}^{-1}(C, K)$ is soft $\aur$-closed, so $\cla(f_{up}^{-1}(C, K)) = f_{up}^{-1}(C, K)$. For arbitrary $(G, K)$, since $(G, K) \softleq \mathrm{cl}_\mathfrak{b}(G, K)$ and $\mathrm{cl}_\mathfrak{b}(G, K)$ is soft $\mathfrak{b}$-closed: $\cla(f_{up}^{-1}(G, K)) \softleq \cla(f_{up}^{-1}(\mathrm{cl}_\mathfrak{b}(G, K))) = f_{up}^{-1}(\mathrm{cl}_\mathfrak{b}(G, K))$.

$(\Leftarrow)$ If $(V, K)$ is soft $\mathfrak{b}$-open, then $(V, K)^c$ is soft $\mathfrak{b}$-closed and $\mathrm{cl}_\mathfrak{b}((V, K)^c) = (V, K)^c$. By hypothesis, $\cla(f_{up}^{-1}((V, K)^c)) \softleq f_{up}^{-1}((V, K)^c)$. Combined with enlargement, $\cla(f_{up}^{-1}((V, K)^c)) = f_{up}^{-1}((V, K)^c)$, so $f_{up}^{-1}(V, K)$ is soft $\aur$-open.
\end{proof}

\section{Separation Axioms}\label{sec:separation}

\begin{definition}\label{def:separation}
Let $(X, \stau, \aur_E, E)$ be a soft aura topological space.
\begin{enumerate}[label=(\alph*)]
\item $(X, \stau, \aur_E, E)$ is \emph{soft $\aur$-$T_0$} if for every pair of distinct points $x, y \in X$, there exists $e \in E$ such that $y \notin \aur_E(x)(e)$ or $x \notin \aur_E(y)(e)$.
\item $(X, \stau, \aur_E, E)$ is \emph{soft $\aur$-$T_1$} if for every pair of distinct points $x, y \in X$ and every $e \in E$, $y \notin \aur_E(x)(e)$ and $x \notin \aur_E(y)(e)$.
\item $(X, \stau, \aur_E, E)$ is \emph{soft $\aur$-$T_2$} (soft $\aur$-Hausdorff) if for every pair of distinct points $x, y \in X$ and every $e \in E$, $\aur_E(x)(e) \cap \aur_E(y)(e) = \emptyset$.
\item $(X, \stau, \aur_E, E)$ is \emph{soft $\aur$-regular} if for every $x \in X$, every $e \in E$, and every soft $\aur$-closed set $(C, E)$ with $x \notin C(e)$, there exist soft $\aur$-open sets $(U, E)$ and $(V, E)$ with $x \in U(e)$, $C(e) \subseteq V(e)$, and $U(e) \cap V(e) = \emptyset$.
\item $(X, \stau, \aur_E, E)$ is \emph{soft $\aur$-$T_3$} if it is soft $\aur$-regular and soft $\aur$-$T_1$.
\end{enumerate}
\end{definition}

The following characterization is fundamental.

\begin{theorem}\label{thm:T1characterization}
$(X, \stau, \aur_E, E)$ is soft $\aur$-$T_1$ if and only if $\aur_E(x)(e) = \{x\}$ for every $x \in X$ and every $e \in E$.
\end{theorem}

\begin{proof}
$(\Leftarrow)$ If $\aur_E(x)(e) = \{x\}$ for all $x, e$, then for $y \neq x$, $y \notin \{x\} = \aur_E(x)(e)$ for every $e$.

$(\Rightarrow)$ Suppose $\aur_E(x)(e) \neq \{x\}$ for some $x$ and $e$. Since $x \in \aur_E(x)(e)$, there exists $y \neq x$ with $y \in \aur_E(x)(e)$. But then $y \in \aur_E(x)(e)$ at this particular $e$, contradicting the $T_1$ requirement that $y \notin \aur_E(x)(e)$ for all $e$.
\end{proof}

\begin{proposition}\label{prop:T1equivT2}
In a soft aura topological space, soft $\aur$-$T_1$ and soft $\aur$-$T_2$ are equivalent.
\end{proposition}

\begin{proof}
Soft $\aur$-$T_2$ $\Rightarrow$ soft $\aur$-$T_1$: If $\aur_E(x)(e) \cap \aur_E(y)(e) = \emptyset$ for all $e$, then since $y \in \aur_E(y)(e)$, we have $y \notin \aur_E(x)(e)$; similarly $x \notin \aur_E(y)(e)$.

Soft $\aur$-$T_1$ $\Rightarrow$ soft $\aur$-$T_2$: By Theorem~\ref{thm:T1characterization}, $\aur_E(x)(e) = \{x\}$ for all $x, e$. Hence $\aur_E(x)(e) \cap \aur_E(y)(e) = \{x\} \cap \{y\} = \emptyset$ for $x \neq y$.
\end{proof}

\begin{remark}\label{rem:T1_collapses}
The equivalence $T_1 \Leftrightarrow T_2$ is in fact a consequence of the scope-based formulation itself: the $T_1$ condition requires $y \notin \aur_E(x)(e)$ for \emph{all} $y \neq x$, which forces every scope to be a singleton, and this mechanism operates identically in both the crisp \cite{AcikgozAura} and soft settings. The genuinely distinctive soft phenomenon lies in the gap between $T_0$ and $T_1$: in the crisp case, $T_0$ requires asymmetric separation for each pair of points via a single scope function, whereas in the soft case, $T_0$ requires separation at merely \emph{some} parameter $e \in E$ while $T_1$ demands it at \emph{every} parameter simultaneously. This parameter-dependent gap makes the $T_0 \not\Rightarrow T_1$ separation substantially wider in the soft setting, as illustrated by Example~\ref{ex:T0notT1}.
\end{remark}

\begin{theorem}\label{thm:sepimplications}
Soft $\aur$-$T_3$ $\Rightarrow$ soft $\aur$-$T_2$ $\Leftrightarrow$ soft $\aur$-$T_1$ $\Rightarrow$ soft $\aur$-$T_0$.
\end{theorem}

\begin{proof}
The equivalence $T_1 \Leftrightarrow T_2$ is Proposition~\ref{prop:T1equivT2}. The implication $T_1 \Rightarrow T_0$ is immediate (the condition for all $e$ implies the condition for some $e$). For $T_3 \Rightarrow T_2$: soft $\aur$-$T_3$ includes soft $\aur$-$T_1$, which is equivalent to soft $\aur$-$T_2$.
\end{proof}

\begin{theorem}\label{thm:T1closure}
If $(X, \stau, \aur_E, E)$ is soft $\aur$-$T_1$, then $\cla(\{x\}_E) = \{x\}_E$ for every $x \in X$.
\end{theorem}

\begin{proof}
By enlargement, $\{x\}_E \softleq \cla(\{x\}_E)$. Conversely, for $y \neq x$ and any $e \in E$: $\aur_E(y)(e) = \{y\}$ by Theorem~\ref{thm:T1characterization}, so $\aur_E(y)(e) \cap \{x\} = \emptyset$, giving $y \notin \cla(\{x\}_E)(e)$.
\end{proof}

\begin{example}\label{ex:T0notT1}
\emph{Soft $\aur$-$T_0$ but not soft $\aur$-$T_1$:} Let $X = \{x_1, x_2\}$, $E = \{e_1, e_2\}$, and let $\stau$ be the discrete soft topology. Define:
\[
\aur_E(x_1)(e_1) = \{x_1\},\; \aur_E(x_1)(e_2) = \{x_1, x_2\};\quad \aur_E(x_2)(e_1) = \{x_1, x_2\},\; \aur_E(x_2)(e_2) = \{x_2\}.
\]
At $e_1$: $x_2 \notin \aur_E(x_1)(e_1) = \{x_1\}$, so $T_0$ is satisfied. However, at $e_2$: $x_2 \in \aur_E(x_1)(e_2) = \{x_1, x_2\}$, so $T_1$ fails. This example also shows that the scope function can carry fundamentally different geometric information at different parameters.
\end{example}

\begin{theorem}\label{thm:scopedependence}
The separation properties depend on the choice of soft scope function. On a fixed soft topological space $(X, \stau, E)$ with $|X| \geq 2$:
\begin{enumerate}[label=(\roman*)]
\item The trivial soft scope function $\aur_E(x) = \softabs$ gives a space that is not soft $\aur$-$T_0$.
\item If $\stau$ contains enough soft open sets, a suitably chosen $\aur_E$ with singleton scopes makes the space soft $\aur$-$T_2$.
\end{enumerate}
\end{theorem}

\begin{proof}
(i) $\aur_E(x)(e) = X$ for all $x, e$, so $y \in \aur_E(x)(e)$ for all $y$, violating $T_0$.

(ii) If for each $x$ there exists $(U_x, E) \in \stau$ with $U_x(e) = \{x\}$ for all $e$, set $\aur_E(x) = (U_x, E)$.
\end{proof}

\section{Soft Aura Rough Approximation}\label{sec:rough}

\begin{definition}\label{def:roughapprox}
Let $(X, \stau, \aur_E, E)$ be a soft aura topological space and $(G, E) \in \mathrm{SS}(X, E)$.
\begin{enumerate}[label=(\alph*)]
\item The \emph{soft aura lower approximation}: $\apr(G, E)(e) = \{x \in X : \aur_E(x)(e) \subseteq G(e)\}$.
\item The \emph{soft aura upper approximation}: $\APR(G, E)(e) = \{x \in X : \aur_E(x)(e) \cap G(e) \neq \emptyset\}$.
\item The \emph{soft aura boundary}: $\bnd(G, E)(e) = \APR(G, E)(e) \setminus \apr(G, E)(e)$.
\item The \emph{soft aura accuracy measure}:
\[
\rho_\aur(G, E) = \frac{\sum_{e \in E} |\apr(G, E)(e)|}{\sum_{e \in E} |\APR(G, E)(e)|}, \quad \text{when } \APR(G, E) \neq \softnull.
\]
\end{enumerate}
\end{definition}

\begin{theorem}\label{thm:roughprops}
The soft aura approximation operators satisfy:
\begin{enumerate}[label=(\roman*)]
\item $\apr(G, E) \softleq (G, E) \softleq \APR(G, E)$.
\item $\apr(\softnull) = \APR(\softnull) = \softnull$ and $\apr(\softabs) = \APR(\softabs) = \softabs$.
\item Monotonicity: $(G, E) \softleq (H, E) \Rightarrow \apr(G, E) \softleq \apr(H, E)$ and $\APR(G, E) \softleq \APR(H, E)$.
\item $\APR((G, E) \softcup (H, E)) = \APR(G, E) \softcup \APR(H, E)$.
\item $\apr((G, E) \softcap (H, E)) = \apr(G, E) \softcap \apr(H, E)$.
\item Duality: $\apr((G, E)^c) = (\APR(G, E))^c$ and $\APR((G, E)^c) = (\apr(G, E))^c$.
\item $0 \leq \rho_\aur(G, E) \leq 1$, with equality to $1$ iff $\bnd(G, E) = \softnull$.
\end{enumerate}
\end{theorem}

\begin{proof}
All properties follow from Theorems~\ref{thm:closureprops}--\ref{thm:duality}, since $\apr = \inte$ and $\APR = \cla$.
\end{proof}

\begin{remark}\label{rem:pawlak}
When the soft scope function arises from an equivalence relation $R$ via $\aur_E(x)(e) = [x]_R$ (constant in $e$), the operators reduce to Pawlak's approximations \cite{Pawlak}. The soft aura model generalizes Pawlak's framework without requiring an equivalence relation, while the parameterization provides multi-criteria resolution.
\end{remark}

\section{Application: Environmental Risk Assessment}\label{sec:application}

We illustrate the soft aura rough approximation framework with an environmental risk assessment problem.

\textbf{Problem setting.} Let $X = \{s_1, s_2, s_3, s_4, s_5\}$ represent five environmental monitoring stations. The parameter set $E = \{e_1, e_2, e_3, e_4\}$ represents environmental quality indicators: $e_1$ = PM$_{2.5}$ concentration, $e_2$ = SO$_2$ concentration, $e_3$ = water pH level, and $e_4$ = dissolved oxygen (DO) level.

\textbf{Soft scope function.} Each station's ``environmental influence zone'' is modeled by a soft scope function $\aur_E$:
\begin{center}
\begin{tabular}{c|cccc}
\toprule
$\aur_E$ & $e_1$ (PM$_{2.5}$) & $e_2$ (SO$_2$) & $e_3$ (pH) & $e_4$ (DO) \\
\midrule
$s_1$ & $\{s_1, s_2\}$ & $\{s_1, s_4\}$ & $\{s_1, s_2\}$ & $\{s_1, s_3\}$ \\
$s_2$ & $\{s_2, s_4\}$ & $\{s_2, s_3\}$ & $\{s_2, s_5\}$ & $\{s_2, s_3\}$ \\
$s_3$ & $\{s_3, s_5\}$ & $\{s_3, s_5\}$ & $\{s_3, s_5\}$ & $\{s_3, s_4\}$ \\
$s_4$ & $\{s_1, s_4\}$ & $\{s_4, s_5\}$ & $\{s_3, s_4\}$ & $\{s_4, s_5\}$ \\
$s_5$ & $\{s_4, s_5\}$ & $\{s_1, s_5\}$ & $\{s_2, s_5\}$ & $\{s_1, s_5\}$ \\
\bottomrule
\end{tabular}
\end{center}

One verifies $s_i \in \aur_E(s_i)(e_j)$ for all $i, j$.

\textbf{Target classification.} Let $(G, E)$ represent ``at-risk'' stations that fail quality standards:
\[
G(e_1) = \{s_3, s_5\},\quad G(e_2) = \{s_2\},\quad G(e_3) = \{s_1, s_4\},\quad G(e_4) = \{s_4, s_5\}.
\]

\textbf{Lower approximation} $\apr(G, E)(e) = \{x : \aur_E(x)(e) \subseteq G(e)\}$:
\begin{itemize}
\item $e_1$: $\aur_E(s_3)(e_1) = \{s_3, s_5\} \subseteq \{s_3, s_5\}$; no other scope is contained. $\apr(e_1) = \{s_3\}$.
\item $e_2$: No scope $\subseteq \{s_2\}$. $\apr(e_2) = \emptyset$.
\item $e_3$: No scope $\subseteq \{s_1, s_4\}$. $\apr(e_3) = \emptyset$.
\item $e_4$: $\aur_E(s_4)(e_4) = \{s_4, s_5\} \subseteq \{s_4, s_5\}$. $\apr(e_4) = \{s_4\}$.
\end{itemize}

\textbf{Upper approximation} $\APR(G, E)(e) = \{x : \aur_E(x)(e) \cap G(e) \neq \emptyset\}$:
\begin{itemize}
\item $e_1$: $s_3$ ($\{s_3, s_5\} \cap \{s_3, s_5\} \neq \emptyset$), $s_5$ ($\{s_4, s_5\} \cap \{s_3, s_5\} = \{s_5\}$). $\APR(e_1) = \{s_3, s_5\}$.
\item $e_2$: $s_2$ ($\{s_2, s_3\} \cap \{s_2\} = \{s_2\}$). $\APR(e_2) = \{s_2\}$.
\item $e_3$: $s_1$ ($\{s_1, s_2\} \cap \{s_1, s_4\} = \{s_1\}$), $s_4$ ($\{s_3, s_4\} \cap \{s_1, s_4\} = \{s_4\}$). $\APR(e_3) = \{s_1, s_4\}$.
\item $e_4$: $s_3$ ($\{s_3, s_4\} \cap \{s_4, s_5\} = \{s_4\}$), $s_4$ ($\{s_4, s_5\} \cap \{s_4, s_5\}$), $s_5$ ($\{s_1, s_5\} \cap \{s_4, s_5\} = \{s_5\}$). $\APR(e_4) = \{s_3, s_4, s_5\}$.
\end{itemize}

\textbf{Summary:}
\begin{center}
\begin{tabular}{c|cccc}
\toprule
& $e_1$ (PM$_{2.5}$) & $e_2$ (SO$_2$) & $e_3$ (pH) & $e_4$ (DO) \\
\midrule
$(G, E)$ & $\{s_3, s_5\}$ & $\{s_2\}$ & $\{s_1, s_4\}$ & $\{s_4, s_5\}$ \\
$\apr(G, E)$ & $\{s_3\}$ & $\emptyset$ & $\emptyset$ & $\{s_4\}$ \\
$\APR(G, E)$ & $\{s_3, s_5\}$ & $\{s_2\}$ & $\{s_1, s_4\}$ & $\{s_3, s_4, s_5\}$ \\
$\bnd(G, E)$ & $\{s_5\}$ & $\{s_2\}$ & $\{s_1, s_4\}$ & $\{s_3, s_5\}$ \\
\bottomrule
\end{tabular}
\end{center}

\textbf{Accuracy measure:}
\[
\rho_\aur(G, E) = \frac{1 + 0 + 0 + 1}{2 + 1 + 2 + 3} = \frac{2}{8} = 0.25.
\]

\textbf{Interpretation.} The accuracy measure $\rho_\aur = 0.25$ indicates significant boundary regions, reflecting genuine uncertainty in the environmental risk classification. Station $s_3$ is \emph{definitively at risk} for PM$_{2.5}$ and station $s_4$ for dissolved oxygen (both appear in the lower approximation). The remaining classifications involve uncertainty due to overlapping environmental influence zones. The parameter-wise analysis identifies which indicators contribute most to uncertainty (pH and SO$_2$ have empty lower approximations), guiding targeted investigation and resource allocation.

\section{Conclusion}\label{sec:conclusion}

We have introduced soft aura topological spaces by equipping a soft topological space with a soft scope function. The main contributions are:

The soft aura-closure operator $\cla$ is a soft additive \v{C}ech closure operator whose transfinite iteration yields a soft Kuratowski closure $\cla^\infty$, generating a chain of soft topologies $\stau_\aur^\infty \subseteq \stau_\aur$.

Five classes of generalized soft open sets form the complete hierarchy: soft $\aur$-open $\Rightarrow$ soft $\aur$-$\alpha$-open $\Rightarrow$ soft $\aur$-semi-open / soft $\aur$-pre-open $\Rightarrow$ soft $\aur$-$b$-open $\Rightarrow$ soft $\aur$-$\beta$-open.

A distinctive discovery is the collapse of soft $\aur$-$T_1$ and soft $\aur$-$T_2$: the scope-based $T_1$ condition forces all scopes to be singletons, automatically yielding $T_2$. The genuinely soft phenomenon is the wide gap between $T_0$ and $T_1$, created by the existential-versus-universal quantification over parameters.

Soft aura rough approximation operators generalize both the crisp aura model \cite{AcikgozAura} and the classical Pawlak model \cite{Pawlak}. The environmental risk assessment application demonstrates parameter-dependent classification with practical interpretive value.

Several directions remain open: soft aura-compactness and connectedness; combining the soft scope function with a soft ideal; extending to intuitionistic fuzzy soft and neutrosophic soft settings; and developing comprehensive decision-making algorithms based on soft aura rough approximations.


\end{document}